\documentclass{daj}
\newif\ifpublic
\publicfalse %%%% Author's notes and TODO's displayed

\ifpublic
\usepackage[disable]{todonotes}
\else
\usepackage[colorinlistoftodos]{todonotes}
\fi

\usepackage{amssymb,amsmath,amsthm,bbm}

\usepackage{amssymb,verbatim,amsthm,amsmath}
\theoremstyle{plain}

\DeclareMathOperator{\Per}{Per}

\newcommand{\Hole}{\textrm{Hole}}
\usepackage{graphicx,color}
\usepackage{hyperref}
\usepackage{enumerate}
\usepackage{xcolor}

\definecolor{olive}{rgb}{0.3, 0.4, .1}
\definecolor{fore}{RGB}{249,242,215}
\definecolor{back}{RGB}{51,51,51}
\definecolor{title}{RGB}{255,0,90}
\definecolor{dgreen}{rgb}{0.,0.4,0.}
\definecolor{gold}{rgb}{1.,0.84,0.}
\definecolor{JungleGreen}{cmyk}{0.99,0,0.52,0}
\definecolor{bgreen}{cmyk}{0.85,0,0.33,0}
\definecolor{RawSienna}{cmyk}{0,0.72,1,0.45}
\definecolor{Magenta}{cmyk}{0,1,0,0}

\newtheorem{theorem}{Theorem}
\newtheorem{cor}[theorem]{Corollary}
\newtheorem{lemma}[theorem]{Lemma}
\newtheorem{prop}[theorem]{Proposition}

\theoremstyle{definition}
\newtheorem*{defn}{Definition}

\newtheorem{claim}[theorem]{Claim}

\newtheorem{corollary}[theorem]{Corollary}

\newtheorem{conjecture}[theorem]{Conjecture}

\theoremstyle{remark}

\newtheorem{remark}[theorem]{Remark}

\newcommand{\eps}{\epsilon}
\newcommand{\G}{{\cal G}}

\newcommand{\Z}{\mathbb{Z}}

\newcommand{\supp}{\textrm{supp}}

\newcommand{\remove}[1]{}

\dajAUTHORdetails{%
  title = {Geometric Stability via Information Theory}, %% please capitalize all significant words
  author = {David Ellis, Ehud Friedgut, Guy Kindler and Amir Yehudayoff},
  plaintextauthor = {David Ellis, Ehud Friedgut, Guy Kindler and Amir Yehudayoff},
  keywords = {Loomis-Whitney inequality, projections, stability, information theory, entropy},
}   %%% END \dajAUTHORdetails

\dajEDITORdetails{%
   year={2016},
   number={10},
   received={29 November 2015},   % received date: example: 7 January 2017
   revised={24 June 2016},    % Optional revised date (you may comment it out)
   published={28 June 2016},  % published date
   doi={10.19086/da.784},       % XXX = number of paper, e.g. da006 for paper#6
}   %%% END \dajEDITORdetails

\begin{document}

\begin{frontmatter}[classification=text]
%% EDITOR: this will force the keywords to appear right after the Abstract.
%%   If the abstract is too long and would force the keywords off the
%%   front page, please comment out % [classification=text] above
%%   This way the keywords will be floated on the bottom of the first page
%%   even though the Abstract spills over to the next page.

%%% AUTHOR: Title goes here.  This line is optional.  You must use it
%%   if title has footnote attached or requires nontrivial typesetting,
%%   e.g., inclusion of linebreaks to force nice layout.
%\title{Short Proof of R\"odl's $n^{\log\log n}$ Bound\footnote{This is a footnote to the title}} %% please capitalize all significant words

%%% AUTHOR:
%%% List all authors. If you wish, place grant acknowledgements in \thanks.
%%% In brackets include a short tag for each author.
\author[david]{David Ellis}
\author[ehud]{Ehud Friedgut\thanks{Supported by I.S.F. grant 1168/15 and by Minerva grant 712023.}}
\author[guy]{Guy Kindler}
\author[amir]{Amir Yehudayoff\thanks{Supported by I.S.F. grant 503/11}}

%%% AUTHOR: Abstract goes here
\begin{abstract}
The Loomis-Whitney inequality, and the more general Uniform Cover inequality, bound the volume of a body in terms of a product of the volumes of lower-dimensional projections of the body. In this paper, we prove stability versions of these inequalities, showing that when they are close to being tight, the body in question is close in symmetric difference to a box. Our results are best possible up to a constant factor depending upon the dimension alone. Our approach is information theoretic.

We use our stability result for the Loomis-Whitney inequality to obtain a stability result for the edge-isoperimetric inequality in the infinite $d$-dimensional lattice. Namely, we prove that a subset of $\mathbb{Z}^d$ with small edge-boundary must be close in symmetric difference to a $d$-dimensional cube. Our bound is, again, best possible up to a constant factor depending upon $d$ alone.
\end{abstract}
\end{frontmatter}

%%% AUTHOR: body of paper starts here
\section{Introduction}
\label{sec:intro}

In this paper, we prove stability results for the Loomis-Whitney
inequality and some of its generalisations. Let us start by describing
these results. 
\subsection{Projection inequalities}
\paragraph{The Loomis-Whitney inequality.}

The Loomis-Whitney inequality~\cite{loomis1949} bounds the volume of a
$d$-dimensional body in terms of the volumes of its $(d-1)$-dimensional projections. It states
that every body $S$ in $\mathbb{R}^d$ satisfies
\begin{equation} \label{eq:lw} \mu(S)^{d-1} \leq \prod_{i \in [d]}
  \mu(\pi_{[d] \setminus \{i\}} (S)),
\end{equation}
where $\mu$ denotes Lebesgue measure, $[d]:=\{1,2,\ldots,d\}$, and for
$g \subset [d]$, we denote by $\pi_g(S)$ the projection of $S$ onto the coordinates of $g$ (i.e., the projection onto the subspace $\{x \in \mathbb{R}^d:\ x_i=0\ \forall i \notin g\}$).  Here, we say that $S \subset \mathbb{R}^d$ is a {\em body} if it is an open set with compact closure. Note that if $S$ is a Cartesian product of subsets of
$\mathbb{R}$, then equality holds in \eqref{eq:lw}; we call such a body a {\em box}.

\paragraph{The Box Theorem and the Uniform Cover inequality.} The Box Theorem of Bollob\'as and Thomason \cite{bollobas-thomason} is a simultaneous generalisation and strengthening of the Loomis-Whitney inequality.  It states that for any body $S \subset
\mathbb{R}^d$, there exists a box $B \subset \mathbb{R}^d$ with the
same measure as $S$, such that $\mu(\pi_{g}(S)) \geq \mu(\pi_{g}(B))$
for all $g \subset [d]$. Bollob\'as and Thomason show that this is equivalent to the so-called `Uniform Cover inequality' of Chung, Frankl, Graham and Shearer~\cite{chung1986some}. We say that a family $\mathcal{G} \subset \mathcal{P}([d])$ is a {\em uniform $m$-cover} if every $i \in [d]$ belongs to exactly
$m$ of the sets in $\G$, and that $\mathcal{G}$ is a {\em uniform cover} if $\mathcal{G}$ is a uniform $m$-cover for some $m \in \mathbb{N}$. The Uniform Cover inequality states that for any body $S \subset \mathbb{R}^d$, any $m \in \mathbb{N}$, and any uniform $m$-cover $\mathcal{G} \subset \mathcal{P}([d])$, we have
\begin{equation}\label{eq:UC}
\mu(S)^{m} \leq \prod_{g \in \G} \mu(\pi_g (S)).
\end{equation}
The Uniform Cover inequality is sharp when $S$ is a box.

Applying the Uniform Cover inequality to sets which are unions of axis-parallel unit cubes, implies that for any finite $S \subset \mathbb{Z}^d$, any $m \in \mathbb{N}$, and any uniform $m$-cover $\G \subset \mathcal{P}([d])$, we have
\begin{equation}\label{eq:UC-finite}
|S|^{m} \leq \prod_{g \in \G} |\pi_g (S)|.
\end{equation}
(Here, of course, $|S|$ denotes the cardinality of the set $S$.) Since increasing the size of a set $g \in \mathcal{G}$ can only increase the right-hand side of (\ref{eq:UC-finite}), it follows that for any set family $\mathcal{G} \subset \mathcal{P}([d])$, and any finite $S \subset \mathbb{Z}^d$, we have
\begin{equation}
\label{eq:UC-finite-extended}
|S|^{m(\G)} \leq \prod_{g \in \G} |\pi_g (S)|,
\end{equation}
where $m(\G)$ denotes the minimum integer $m$ such that every $i \in [d]$ belongs to at least $m$ of the sets in $\mathcal{G}$.

In fact, \eqref{eq:UC-finite} implies \eqref{eq:UC}, by a standard approximation argument, approximating a body $S \subset \mathbb{R}^d$ by a union of cubes in a sufficiently fine grid, as outlined in \cite{loomis1949}. Note, however, that the analogue of \eqref{eq:UC-finite-extended} for bodies in $\mathbb{R}^d$ (with Lebesgue measure) does not necessarily hold if $\mathcal{G}$ is not a uniform cover, as can be seen by taking $S$ to be a $d$-dimensional axis-parallel cube of side-length less than 1. (Roughly speaking, the approximation argument requires both sides of (\ref{eq:UC}) to scale by the same factor, when $S$ is dilated; see the proof of Corollary \ref{corr:uniform-cover-stability} in Section \ref{Lemmas} below.)

% Roughly speaking, uniformity is needed to ensure
% that if we blow up a body $S \subset \R^d$ by a constant factor,
% the two sides of \eqref{eq:UC} change by the same amount.
% This issue comes up in the approximation argument,
% since a box of size length $\delta$ has volume $\delta^d$
% in $d$-dimensional space.

%, as can be seen by taking $S$ to be a $d$-dimensional axis-parallel cube of side-length less than 1. \dnote{Added to address the referee's comment (1).}

\paragraph{Shearer's Lemma.} Loomis and Whitney, and Bollob\'as and Thomason, 
 proved their results using induction on the dimension, and H\"{o}lder's inequality. However, the discrete versions of the Loomis-Whitney and Uniform Cover inequalities (which are equivalent to the continuous ones) are special cases of Shearer's Entropy Lemma. The term `Shearer's Lemma' is often used to refer to two essentially equivalent results. The first is stated in terms of the entropy of a random variable, first proved by Shearer (implicitly) in 1978, and first published by Chung, Frankl, Graham and Shearer in 1986 in \cite{chung1986some}. The second, from the same paper, is equation \eqref{eq:UC}, the result we referred to above as the Uniform Cover inequality. The entropy approach for proving these inequalities is the one we adopt in this paper.

\bigskip

As well as being very natural statements in their own right, the above results have many applications: for example in convex geometry (see \cite{ball-handbook}), in the study of isoperimetric problems (see \cite{loomis1949}), in extremal combinatorics (see \cite{chung1986some}) and in additive combinatorics (see \cite{ballister-bollobas}).
There are also several useful generalisations of Shearer's Lemma, such as the weighted version in~\cite{Friedgut04}, which is itself a special case of the Brascamp-Lieb inequality~\cite{brascamp-lieb}. In~\cite{Ball89}, Ball stated and applied the geometric version of the Brascamp-Lieb inequality to study sections of the Euclidean cube. The results in this paper apply directly to several of these generalisations. 

\subsubsection*{Stability versions}

When a geometric, combinatorial or functional inequality is sharp, it
is natural to ask whether it is also `stable' --- i.e., when the inequality is
{\em almost} sharp for a particular object, must that
object be close in structure to an {\em extremal} one (meaning, an object for which equality holds)?

Stability phenomena for geometric and functional inequalities have
been widely studied in recent years. To obtain a stability result for
an inequality, it is natural to look closely at known proofs of the inequality and
see what information these proofs yield about objects where equality
`almost' holds. Several methods for proving geometric inequalities
have recently been shown to yield best-possible (or close to
best-possible) stability results. A partial list includes symmetrization
techniques (see e.g. \cite{fusco2008sharp}), optimal transport (see
e.g. \cite{figalli2010mass,cicalese}), spectral techniques (see
e.g. \cite{bianchi}), and non-linear evolution equations (see
e.g. \cite{dolbeault}). Stability phenomena for combinatorial
inequalities have also been widely studied, and best-possible (or
close to best-possible) stability results have been obtained via
elementary combinatorial arguments (see
e.g. \cite{simonovits,frankl}), and using spectral techniques (see
e.g. \cite{friedgut-tintersecting,deza-frankl-stability}), Fourier
analysis (see e.g. \cite{friedgut-tintersecting,friedgut-junta}) and
`non-Abelian' Fourier analysis (see
e.g. \cite{deza-frankl-stability,EFF1,EFF2}).

Our main result in this paper is a stability result for the Uniform Cover inequality. To state it, we need some more notation. We define a {\em box} in $\mathbb{Z}^d$ to be a Cartesian product of finite subsets of $\mathbb{Z}$. For a collection
of subsets of coordinates $\mathcal{G} \subset \mathcal{P}([d])$, we
denote by $\sigma(\G)$ the maximum integer $\sigma$ such that for every $i,j \in [d]$ with $i \neq j$,
there are at least $\sigma$ sets in $\G$ containing $i$ but not $j$. If $\sigma(\G) >0$, then we define
$$\rho(\G) := \frac{m(\G)}{\sigma(\G)}.$$
We prove the following stability result for the inequality (\ref{eq:UC-finite-extended}).
\begin{theorem}
\label{thm:uniform-cover-stability}
For every integer $d \ge 2$ there exists $b = b(d)>0$ such that the following holds. Let $\mathcal{G} \subset \mathcal{P}([d])$ with $m(\G), \sigma(\G) > 0$.  Let $S \subset \mathbb{Z}^d$ with $|S| < \infty$. If 
$$|S| \geq (1-\eps) \left( \prod_{g \in \G} |\pi_g (S)| \right)^{1/m(\G)},$$
then there exists a box $B \subset \mathbb{Z}^d$ such that
$$|S \Delta B| \leq b  \rho(\G) \eps \,|S|.$$
\end{theorem}
Our proof yields $b(d) = 4d^2+64d$. This theorem is best-possible in terms of its dependence upon $\epsilon$, as can be seen by taking $S = [a]^d \setminus [a']^d$, where $(a'/a)^d = \epsilon < 2^{-d}$, and taking $\G = [d]^{(d-1)}$.

Theorem \ref{thm:uniform-cover-stability} easily implies the following analogue for bodies (with uniform covers), via the standard approximation argument outlined in \cite{loomis1949}, and referred to above.

\begin{corollary}
\label{corr:uniform-cover-stability}
For every integer $d \ge 2$ there exists $b = b(d) > 0$ such that the following holds. Let $m \in \mathbb{N}$, and let $\G \subset \mathcal{P}([d])$ be a uniform $m$-cover with $\sigma(\G)>0$. Let $S \subset \mathbb{R}^d$ be a body such that
\begin{equation}\label{eq:tightness2} \mu(S) \geq (1-\eps) 
 \left( \prod_{g \in \G} \mu(\pi_g (S)) \right)^{1/m}.
\end{equation}
Then there exists a box $B \subset \mathbb{R}^d$ such that 
$$\mu(S \Delta B) \leq b \rho(\G) \eps\, \mu(S).$$
\end{corollary}

 For completeness, we present in full the deduction of Corollary \ref{corr:uniform-cover-stability} from Theorem \ref{thm:uniform-cover-stability}, in Section \ref{Lemmas}.

If $\mathcal{G} = [d]^{(d-1)}$, then $m(\G) = d-1$, $\sigma(\G)=1$ and $\rho(\G) = d-1$, so the following stability result for the Loomis-Whitney inequality, is a special case of Theorem \ref{thm:uniform-cover-stability} .
\begin{cor}
\label{cor:lw-stability}
For every integer $d \ge 2$ there exists $c = c(d)>0$ such that the following holds. Let $S \subset \mathbb{Z}^d$ with $|S| < \infty$. If 
$$|S| \geq (1-\eps) \left( \prod_{i=1}^{d} |\pi_{[d]\setminus\{i\}} (S)| \right)^{1/(d-1)},$$
then there exists a box $B \subset \mathbb{Z}^d$ such that
$$|S \Delta B| \leq c \eps \,|S|.$$
\end{cor}

Of course, we can take $c(d) = (d-1)b(d) = (d-1)(4d^2+64d) \leq 36d^3$ in this corollary.

\paragraph*{Stability for a weighted version of the Uniform Cover inequality.}
In \cite{Friedgut04}, a weighted version of the Uniform Cover inequality is proved, a version that is, in fact, a special case of the Brascamp-Lieb inequality of \cite{brascamp-lieb}. It is not hard to verify that our proof in this paper goes through almost word for word to yield the following stability result for the weighted version.   
\begin{theorem}\label{thm:weighted}
For every integer $d \ge 2$ there exists $b = b(d) > 0$ such that the following holds. Let $\G \subset \mathcal{P}([d])$, and let $w: \G \rightarrow \mathbb{R}_{\geq 0}$ be a non-negative weight function on $\G$, such that every $i \in [d]$ is covered by sets with total weight at least 1, i.e.
$$
\sum_{g \in \G:\atop i \in g} w(g) \geq 1.
$$
Let $\sigma(\G):= \min_{i\not = j} \sum_{g \cap \{i,j\} = \{i\}} w(g),$ assume $\sigma(\G) >0$, and let $\rho(\G) = 1/\sigma(\G)$. Let $S \subset \mathbb{Z}^d$ with $|S| < \infty$.
If 
$$
|S| \geq (1-\eps)  \prod_{g \in \G}  |\pi_g (S)| ^{w(g)}
$$
then
there exists a box $B \subset \mathbb{Z}^d$ such that
$$
|S \Delta B| \leq b  \rho(\G) \eps \,|S|.
$$
\end{theorem} 
Again, we can take $b(d) = 4d^2+64d$.
\begin{comment} Of course, setting $\eps=0$ gives the original version of the theorem (even if $\sigma(\G)=0$). (David: do we really need to say this? I'd prefer not as it seems a bit clumsy.) \end{comment}
We omit the details of the proof of Theorem \ref{thm:weighted}.
%For brevity, we introduce the following terminology. If $S \subset \mathbb{Z}^d$ is finite, and $\G \subset \mathcal{P}([d])$, we say that $S$ is {\em $\epsilon$-tight for $\mathcal{G}$} if
%$$|S| \geq (1-\eps) \left( \prod_{g \in \G} |\pi_g S| \right)^{1/m(\G)}.$$
%Note that $S$ satisfies hypothesis (\ref{eq:tightness}) of Theorem \ref{thm:loomis-whitney-stability} iff it is $\epsilon$-tight for $[d]^{(d-1)}$.

\paragraph*{About the proof.} 

A few words regarding our proof of Theorem~\ref{thm:uniform-cover-stability}. 
To prove a stability result for some inequality, it is natural to 
consider a certain proof of that inequality, and
to `work backwards' through this proof to deduce closeness to the desired structure when equality almost holds. This may be called a `stable version' of the relevant proof. Perhaps the first natural approach to proving Theorem \ref{thm:uniform-cover-stability} is to produce a stable version of the classical proof of \eqref{eq:UC-finite-extended}, which uses H\"{o}lder's inequality and induction on the dimension. However, even for the Loomis-Whitney inequality, this only yields $\Theta(\sqrt{\epsilon})$-dependence. In order to obtain the sharp $\Theta(\epsilon)$-dependence we seek, we consider the beautiful
Llewellyn-Radhakrishnan proof of Shearer's Lemma (see
\cite{radhakrishnan}).
This proof is information-theoretic, using simple properties of entropy.

Given a set $S \subset \Z^d$ and a family $\G \subset \mathcal{P}([d])$, such that the inequality (\ref{eq:UC-finite-extended}) is almost tight for $S$ and $\G$, we first deduce that, in a sense, $S$ is `close' to a `product structure';
that is, the uniform distribution on $S$ is `close' to the product of its marginals (the relevant notion of distance is discussed below). That part of the argument is 
a fairly straightforward deduction from the entropy proof of Llewellyn-Radhakrishnan, or alternatively from the entropy inequalities proved by Balister and Bollob\'as in \cite{ballister-bollobas}. (It is also inspired by the proof of the Parallel
Repetition Theorem in \cite{Raz98}, and related works such as
\cite{BarakBCR10}.)

Although information theory allows us to
easily identify a product structure,
it moves us from the language of sets
to the language of distributions.
The next ingredient of the proof is more combinatorial.
Its purpose is to move back from the language of distributions
to the language of sets; we identify the actual box that we claim exists.

It turns out that for this part of the proof, a `two-dimensional lemma' suffices.
Given a set $S \subset X_1 \times X_2$, such that the uniform distribution on $S$ is `close' to the product of its marginals, we find a two-dimensional box (a `rectangle') $R_1 \times R_2$, which is a good approximation of the set $S$.
To prove this lemma, we need to identify the two sets $R_1 \subset \pi_1(S)$
and $R_2 \subset \pi_2(S)$.
This is done via an iterative `trimming' procedure,
which gradually removes parts of $\pi_1(S)$ and $\pi_2(S)$,
until only $R_1$ and $R_2$ remain.
The crux of the proof is in showing that we did not throw too much,
i.e., that $S \setminus (R_1 \times R_2)$ is small.

We apply this two-dimensional lemma $d$ times; for each $i \in [d]$ we consider our set $S \subset \Z^d$ as a two-dimensional set $S \subset \Z^{\{i\}} \times \Z^{[d] \setminus \{i\}}$, and we find a set $R_i \subset \Z^{\{i\}}$ which is a good candidate to be the `edge in direction $i$' of the box approximating $S$. We then check that, indeed, $S$ is close to the box
$R_1 \times R_2 \times \ldots \times R_d$.

The last issue to discuss is the meaning of the word `close' above:
how to measure the `distance'
between the uniform distribution on $S$ and the product of 
its marginals.
The information theoretic part of the argument
naturally leads to measuring this distance 
using the Kullback-Leibler divergence.
A natural thing to do at this point would be to use Pinsker's
inequality to move from Kullback-Leibler divergence to $\ell^1$-distance. 
This, however, leads to suboptimal $\Theta(\sqrt{\eps})$-dependence.
To overcome this difficulty, we introduce a new (but natural) measure
of distance, which we call the `hole-weight'.
The hole-weight suffices to control the trimming procedure,
yielding optimal $\Theta(\eps)$-dependence.
It may find applications in other, similar scenarios. For more details on this part of the proof,
see Section~\ref{Lemmas} below.

\subsection{Isoperimetric inequalities}
In Section \ref{sec:edgeIso}, we apply Theorem
\ref{thm:uniform-cover-stability} to prove another result, demonstrating
stability for the edge-isoperimetric inequality in the infinite
$d$-dimensional lattice. Before stating this formally, we give some background
on isoperimetric inequalities.

Isoperimetric problems are classical objects of study in
mathematics. In general, they ask for the minimum possible `boundary size'
of a set of a given `size', where the exact meaning of these words
varies according to the problem. 

The classical isoperimetric
problem in the plane asks for the minimum possible perimeter of a
shape in the plane with area 1. The answer, that it is best to take a
circle, was already known to the ancient Greeks, but it was not until the
19th century that this was proved rigorously\footnote{The first complete
proof, by placing the calculus of variations on a fully rigorous
footing, was given by Weierstrass in a series of lectures in the
1870s in Berlin.}.

The isoperimetric inequality for Euclidean space states that among all
subsets of $\mathbb{R}^d$ of given volume, Euclidean balls have the
smallest boundary. To state a version of this precisely, if
$A \subset \mathbb{R}^d$ is a Borel set of finite Lebesgue measure, we
denote by $\Per(A)$ the {\em distributional perimeter} of $A$ (see
e.g. Chapter 12 in \cite{maggi2012sets} for the definition of
distributional perimeter).
When $A$ has piecewise smooth topological boundary $\partial A$, then
$\Per(A) = \mu_{d-1}(\partial A)$,
where $\mu_{d-1}(\partial A)$ denotes the $(d-1)$-dimensional Hausdorff measure of $\partial A$, a measure of the boundary which may be more familiar to some readers.

\begin{theorem}
\label{thm:iso-Rd}
If $S \subset \mathbb{R}^d$ is a Borel set with Lebesgue measure $\mu(S) < \infty$, then
$$\Per(S) \geq \Per(B),$$
where $B$ is an Euclidean ball in $\mathbb{R}^d$ with $\mu(B)=\mu(S)$.
\end{theorem}

In this paper, we consider a discrete analogue of 
Theorem~\ref{thm:iso-Rd}. To state it, we
need some more notation. Let $\mathbb{L}^d$ denote the graph of the
$d$-dimensional integer lattice, i.e.\ the graph with vertex-set $\mathbb{Z}^d$
and edge-set
$$\{\{x,x+e_i\}:\ x \in \mathbb{Z}^d,\ i \in [d]\},$$
where $e_i = (0,0,\ldots,0,1,0,\ldots,0)$ denotes the $i$th unit
vector in $\mathbb{R}^d$. If $S \subset \mathbb{Z}^d$, we let
$\partial S$ denote the {\em edge-boundary} of $S$ in the graph
$\mathbb{L}^d$, meaning the set of edges of $\mathbb{L}^d$ which join
a vertex in $S$ to a vertex not in $S$. 

The following
edge-isoperimetric inequality in $\mathbb{L}^d$ is an easy consequence
of the Loomis-Whitney Inequality, and the inequality of arithmetic and geometric means (the AM-GM inequality, for short).
It is also an immediate consequence of Theorem 8 in~\cite{bollobas-leader-grid}.

\begin{theorem}
\label{thm:edge-iso}
Let $S \subset \mathbb{Z}^d$ with $|S| < \infty$. Then
$$|\partial S| \geq 2d|S|^{(d-1)/d}.$$
\end{theorem}

Equality holds in Theorem~\ref{thm:edge-iso} if $S = [a]^d$ for some $a \in \mathbb{N}$. Very slightly more generally, equality holds if $S = S_1 \times S_2 \times \ldots \times S_d$, where $S_1,\ldots,S_d$ are equal-sized intervals in $\mathbb{Z}$; we will call such a set a {\em cube}.

\subsubsection*{Isoperimetric stability}

In their seminal work \cite{fusco2008sharp},
Fusco, Maggi and Pratelli proved that if $S \subset \mathbb{R}^d$ is a
Borel set of finite measure and with distributional perimeter close to
the minimum possible size (viz, the size given by Theorem
\ref{thm:iso-Rd}), then $S$ must be close in symmetric difference to an
Euclidean ball of the same measure, confirming a conjecture of Hall.
\begin{theorem}[Fusco, Maggi, Pratelli]
\label{thm:fmp}
Suppose $S \subset \mathbb{R}^d$ is a Borel set with Lebesgue measure $\mu(S) < \infty$, and with
$$\Per(S) \leq (1+\epsilon)\Per(B),$$
where $B$ is a Euclidean ball with $\mu(B)=\mu(S)$. Then there exists $x \in \mathbb{R}^d$ such that
$$\mu(S \Delta (B + x)) \leq C_d\, \sqrt{\epsilon}\, \mu(S),$$
where $C_d >0$ is a constant depending upon $d$ alone.
\end{theorem}

As observed in \cite{fusco2008sharp}, Theorem \ref{thm:fmp} is sharp up to the value of the constant $C_d$, as can be seen by taking $S$ to be an ellipsoid with $d-1$ semi-axes of length 1 and one semi-axis of length slightly larger than 1.

In this paper, we prove a discrete analogue of Theorem \ref{thm:fmp}
by using Theorem \ref{thm:uniform-cover-stability}. 
We prove the following stability result for the edge-isoperimetric inequality in $\mathbb{L}^d$.

\begin{theorem}
\label{thm:edge-iso-stability}
Let $d \in \mathbb{N}$ with $d \geq 2$. If $S \subset \mathbb{Z}^d$ with $|S| < \infty$ and with
$$|\partial S| \leq (1+\epsilon) 2d|S|^{(d-1)/d},$$
then there exists a cube $C \subset \mathbb{Z}^d$ such that
$$|S \Delta C| \leq 72 d^{5/2} \sqrt{\epsilon} |S|.$$
\end{theorem}

Theorem \ref{thm:edge-iso-stability} has the best possible dependence
on $\epsilon$, as can be seen by taking a `cuboid'
$S = [a]^{d-1} \times [b]$, where $b$ is slightly larger than $a$ (see
Remark \ref{remark:iso-sharpness} for details). We conjecture that the dependence on $d$ could be improved, to $\Theta(\sqrt{d})$ (see section \ref{sec:conclusion}).

\subsection{Structure of paper}

In section \ref{sec:prelim}, we cover some background and introduce
some notation.  In subsection \ref{Lemmas}, we present our main lemmas and prove that they imply our main theorem.
In subsection \ref{ProofsOfLemmas}, we prove the main lemmas. In section
\ref{sec:edgeIso}, we prove Theorem \ref{thm:edge-iso-stability}, our
stability result for the edge-isoperimetric inequality in the lattice
$\mathbb{L}^d$. Finally, in section \ref{sec:conclusion}, we conclude with some
open problems.

\section{Preliminaries} \label{sec:prelim}
In this section, we state some definitions and known results from
probability theory and information theory, and we describe some of our
notation. For background concerning the information-theoretic
results, and for proofs, the reader is referred to
\cite{cover2012elements}.

Throughout this paper, $\log$ means $\log_2$, and we use the convention $0 \log 0 = 0$.

\begin{defn}
Let $p$ be a probability distribution on a finite or countable set $X$. The {\em entropy} of $p$ is defined by
$$H(p) := \sum_{x \in X} p(x) \log (1/p(x)).$$
\end{defn}
The entropy of a random variable is the entropy of its distribution. (By a slight abuse of terminology, we will often identify a random variable with its distribution.)

Intuitively, the entropy of a random variable 
measures the `amount of information' one has from knowing the value of the random variable.

Let $\supp(p)$ denote the support of the distribution $p$, i.e. $\supp(p) = \{x:\ p(x) \neq 0\}$. The convexity of $t \mapsto \log(1/t)$ implies that
\begin{equation}\label{eq:uniform-bound} H(p) \leq \log |\supp(p)|.\end{equation}

Note that equality holds in (\ref{eq:uniform-bound}) iff $p$ is uniformly distributed on its support.
\begin{defn}
For two random variables $A,B$ taking values in a set $X$, the {\em conditional entropy of $A$ given $B$} is defined by
$$H(A|B) := H(A,B) - H(B).$$
\end{defn}

The {\em chain rule} for entropy follows immediately:
$$H(A,B) = H(B) + H(A|B).$$
It is easy to prove that {\em conditioning does not increase entropy}: for any two random variables $A,B$,
$$H(A|B) \leq H(A).$$

For three random variables $A,B,C$,
we denote by $p(a,b,c)$ the probability of the event $\{A=a,B=b,C=c\}$, we denote by $p(a)$ the probability of the event $\{A=a\}$, and if $p(b)>0$, we denote by $p(a|b)$ the probability of $\{A=a\}$ given
$\{B=b\}$.

\begin{defn}
If $p$ and $q$ are two distributions on a finite or countable set $X$, with $\supp(p) \subset \supp(q)$, the {\em Kullback-Leibler divergence} between $p,q$ is defined by
\[D(p||q) := \sum_x p(x) \log \frac{p(x)}{q(x)}.\]
If $\supp(p) \not\subset \supp(q)$, we define $D(p||q) := \infty$.
\end{defn}

The Kullback-Leibler divergence is non-negative, it is zero if and only if $p=q$, but it is not symmetric, i.e. in general, $D(q||p) \neq D(p||q)$, even when $\supp(p)=\supp(q)$.

\begin{defn}
The {\em mutual information of $A$ and $B$} is defined by
\begin{align}
I(A;B) = H(A) - H(A|B)
= \sum_{a,b} p(a,b) \log \frac{p(a,b) }{p(a)p(b)}.
\end{align}
Note that this is also the Kullback-Leibler divergence between the joint distribution of $A$ and $B$ and the product of the marginals.

The {\em mutual information of $A$ and $B$, conditioned on $C$}, is defined by
\begin{align}
\label{eqn:info}
I(A;B\,|\,C) = H(A|C) - H(A\,|\,(B,C))
= \sum_{a,b,c} p(a,b,c) \log \frac{p(a,b|c) }{p(a|c)p(b|c)}.
\end{align}
\end{defn}
Mutual information {\em is} symmetric under interchanging $A$ and $B$, i.e.
$$I(B;A\,|\,C) = I(A;B\,|\,C).$$
Another (intuitively plausible) property of mutual information is that if $C$ is a function of $B$,
then 
$$
I(A;C) \leq I(A;B).
$$
We refer to this property as the `monotonicity of mutual information'.
\paragraph{Marginal distributions.}
Let $p$ be a probability distribution on $\Z^d$. For a subset $g \subset [d]$, we denote by $p_g$ the marginal distribution of $p$
on the set of coordinates $g$, i.e.
$$\forall S \subset \mathbb{Z}^g, \qquad p_g(S) = p(S \times \mathbb{Z}^{[d] \setminus g}).$$
If $(g_1,g_2,\ldots,g_r)$ is a partition of $[d]$,
we denote by $p_{g_1}  p_{g_2}  \ldots p_{g_r}$ 
the obvious product-distribution on $\Z^d$.

We will need the following equation relating the divergences between various products of marginals of $p$.
\begin{equation}\label{telescope}
D\left(p\mid \mid \prod_{i=1}^d p_i\right)= \sum_{i=2}^{d} D\left(p_{[i]} \mid \mid p_{[i-1]}p_i\right).
\end{equation}
This is easily verified, using the definition of $D$ and expanding the logarithms on the right-hand side. Note that if $X = (X_1,X_2,\ldots,X_d)$ is a random variable with probability distribution $p$, then the left-hand side is precisely the {\em total correlation} of the set of random variables $\{X_1,\ldots,X_d\}$.

\section{Proof of Theorem \ref{thm:uniform-cover-stability}}
\subsection{Main lemmas and the deduction of Theorem \ref{thm:uniform-cover-stability}}\label{Lemmas}
In this subsection, we present several statements that, when put together, easily imply our main theorem.

First, we would like to show that if the Uniform Cover inequality is close to being tight for a set $S$, then there is not much mutual information between any 1-dimensional marginal of the uniform distribution on $S$, and the complementary $(d-1)$-dimensional marginal.
\begin{lemma}
\label{lem:TightToInfo}
Let $d \in \mathbb{N}$ with $d \geq 2$. Let $\G \subset \mathcal{P}([d])$ with $m(\G),\sigma(\G)>0$. Let $0 \leq \eps \leq \tfrac{1}{2}$.
Let $S \subset \Z^d$ with $|S| < \infty$ and with
$$|S| \geq (1-\eps) \left( \prod_{g \in \G} |\pi_g (S)| \right)^{1/m(\G)}.$$
Let $p$ denote the uniform distribution on $S$. Then for all $i \in [d]$, we have
$$
 I(p_{\{i\}};p_{[d] \setminus \{i\}}) \leq 2 \rho(\G) \eps .
$$
\end{lemma}
Note that, by monotonicity of mutual information, Lemma \ref{lem:TightToInfo} implies that for any $J \subset  [d]$ and any $i \not \in J$, we have
\begin{equation}\label{eq:SmallD}
D(p_{J \cup \{i\}}||p_Jp_{\{i\}}) = I(p_{\{i\}};p_J) \leq I(p_{\{i\}};p_{[d] \setminus \{i\}}) \leq 2 \rho(\G) \eps.
\end{equation}
Given a set $S \subset \Z^d$, we want to measure how far the uniform distribution $p$ on $S$ is from the product of some of its marginals. It turns out that a useful measure for us (which we call the `hole-weight') is the sum of the product of the marginals over all points {\em not} in $S$ (`holes'). Formally, if $(g_1,\ldots, g_r)$ is a partition of $[d]$, we define the {\em hole-weight} of $S$ with respect to $(g_1,\ldots,g_r)$ by
$$
\Hole_{g_1,\ldots, g_r}(S) := \sum_{x \not \in S} \prod_{j=1}^{r} p_{g_j}(x).
$$
In all but one case below, the partition of $[d]$ which defines the hole-weight will be of the form $(\{i\},[d]\setminus\{i\})$, so for brevity, we write $\Hole_i(S):= \Hole_{\{i\},[d]\setminus\{i\}}(S)$. Also, when stating and proving lemmas, instead of considering $S \subseteq \Z^{\{i\}} \times \Z^{[d] \setminus \{i\}}$, we will sometimes consider the general `two-dimensional' setting $S \subset X_1 \times X_2$, i.e. $S$ is simply a subset of a product of two sets. If $X_1$ and $X_2$ are sets, and $S \subset X_1 \times X_2$, we will write $\Hole(S):=\Hole_{\{1\},\{2\}}(S)$. 

The following claim bounds from above the hole-weight of $S$ by the Kullback-Leibler divergence between the uniform distribution on $S$, and the product of its marginals.
\begin{claim}
\label{clm:dHvsdD}
Let $S \subset \Z^d$ with $|S| < \infty$, let $p$ be the uniform distribution on $S$, and let $(g_1,\ldots g_r)$ be a partition of $[d]$. Let $p$ be the uniform distribution on $S$, and let $(p_{g_j})$ denote the corresponding marginals. Then 
\begin{equation}\label{eq:hole-div}
\Hole_{g_1,\ldots g_r}(S) \leq D\left(p\mid \mid \prod_{j=1}^{r} p_{g_j}\right).
\end{equation}
\end{claim}

\begin{remark}
Of course, another very natural way of measuring how far $p$ is from the product of its marginals is to simply use the $\ell^1$-distance
\begin{equation}
\label{eq:l1}
\|p - \prod_{j=1}^{r} p_{g_j}\|_1.
\end{equation}
(Recall that if $p$ and $q$ are probability distributions on a finite or countable set $X$, then the {\em $\ell^1$-distance} between $p$ and $q$ is defined by
$$\|p-q\|_1 : = \sum_{x \in X} |p(x) - q(x)| = 2 \max\{p(S)-q(S):\ S \subset X\};$$
the quantity $\max\{p(S)-q(S):\ S \subset X\} = \tfrac{1}{2}\|p-q\|_1$ is often called the {\em total variation distance} between $p$ and $q$.)

One can bound the $\ell^1$-distance (\ref{eq:l1}) in terms of the divergence on the right-hand side of (\ref{eq:hole-div}), using Pinsker's inequality. Pinsker's inequality states that if $p$ and $q$ are two probability distributions on a finite or countable set $X$, then
\begin{equation} \label{eq:pinsker} \|p-q\|_1 \leq \sqrt{(2 \ln 2) \, D(p||q)}.\end{equation}
(Note that Pinkser originally proved \eqref{eq:pinsker} with a worse constant. The above form, in which the constant is sharp, is due independently to Kullback, Csisz\'ar and Kemperman.) Applying this yields
$$\|p - \prod_{j=1}^{r} p_{g_j}\|_1 \leq \sqrt{(2 \ln 2)\, D\left(p\mid \mid \prod_{j=1}^{r} p_{g_j}\right)}.$$
Unfortunately, this application of Pinsker's inequality introduces $\Theta(\sqrt{\epsilon})$-dependence in the conclusion of Theorem \ref{thm:uniform-cover-stability} (this was our original approach). We obtain $\Theta(\epsilon)$-dependence by relying only on the hole-weight.
\end{remark}

We now need a lemma saying that if the hole-weight of a two-dimensional set $S$ is small, then $S$ is close to a 2-dimensional box.
\begin{lemma}
\label{lem:FindBox1}
Let $X_1$ and $X_2$ be sets. Let $S \subset X_1 \times X_2$ with $|S| < \infty$. Let $p$ denote the uniform distribution on $S$, and let $p_1,p_2$ denote its marginals. Let $0 < \alpha <1$.
 Then there exists $R_1 \subset X_1$ such that 
$$p_1(X_1\setminus R_1) \leq \frac{2\Hole(S)}{\alpha},$$
and such that for every $x_1 \in R_1$, we have
$$
p_1(x_1) \geq \left(1-\frac{2 \Hole(S)}{\alpha}\right) \cdot 
\frac{(1-\alpha)}{|R_1|}.
$$
\end{lemma}
The idea behind Lemma \ref{lem:FindBox1} is that the set $R_1 \subset X_1$ is a good candidate to be one of the multiplicands (`edges') of a box approximating $S$: on the one hand, it captures most of $p_1(S)$, and on the other hand, $p_1$ restricted to $R_1$ is close (in a sense) to being uniform. Indeed, the main step in the proof of Theorem \ref{thm:uniform-cover-stability} below is to apply Lemma \ref{lem:FindBox1}, with an appropriate value of $\alpha$, yielding (for each $i \in [d])$ a set $R_i \subset \Z^{\{i\}}$, and then to show that $S$ is close in symmetric difference to the Cartesian product of the $R_i$'s.  

The above lemmas now yield the proof of our main theorem.

\begin{proof}[Proof of Theorem \ref{thm:uniform-cover-stability}]
Let $\G \subset \mathcal{P}([d])$ with $m(\G),\sigma(\G) >0$. We may and shall assume that $\eps < ((4d^2+64d)\rho(\G))^{-1}$, as otherwise the conclusion of the theorem holds trivially. Given a set $S \subset \mathbb{Z}^d$ with $|S| < \infty$, and with
$$
|S| \geq (1-\eps) \left( \prod_{g \in \G} |\pi_g (S)| \right)^{1/m(\G)},
$$ we apply Lemma \ref{lem:TightToInfo} to deduce that for all $i \in [d]$,
$$
 I(p_{\{i\}};p_{[d]\setminus \{i\}}) \leq 2 \rho(\G) \eps.
$$
This implies, via (\ref{eq:SmallD}) and Claim \ref{clm:dHvsdD}, that for every $i$ in $[d]$,
$$ 
\Hole_i(S) \leq 2 \rho(\G) \eps.
$$
Next, for each $i \in [d]$, we apply Lemma \ref{lem:FindBox1} to $S \subseteq \Z^{\{i\}} \times \Z^{[d]\setminus \{i\}}$ (i.e., we take $X_1 = \Z^{\{i\}}$ and $X_2 = \Z^{[d]\setminus \{i\}}$), so that $\Hole(S) = \Hole_i(S) \leq 2 \rho(\G) \eps$; we take $\alpha = 1/d$. This yields (for every $i \in [d]$) a set 
$R_i \subset \Z^{\{i\}}$ such that
$$
p_i( \Z^{\{i\}} \setminus R_i) \leq 2 d \cdot \Hole_i(S) \leq 4 d \rho(\G) \eps,
$$
and such that for any $x_i\in R_i$,
$$p_i(x_i) \geq (1-2d\cdot  \Hole_i(S))(1-1/d) \frac{1}{|R_i|} \geq (1- 4 d \rho(\G) \eps ) (1-1/d)\frac{1}{|R_i|}
\geq (1-1/d)^2 \frac{1}{|R_i|}.$$
Let
$$R :=R_1 \times R_2 \times \ldots \times R_d.$$ 
By the union bound,
$$
p( S \setminus R) \leq 4 d^2 \rho(\G) \eps,
$$ 
i.e.
\begin{equation}\label{eq:s-r} |S \setminus R| \leq 4 d^2 \rho(\G) \eps |S|.\end{equation}
On the other hand, for every $x \in R \setminus S$, we have
$$
\prod_{i=1}^d p_i(x_i) \geq \prod_{i=1}^d \frac{(1- 1/d)^{2}}{  |R_i|} = \frac{(1-1/d)^{2d}}{|R|} \geq \frac{1}{16|R|}.
$$
Hence,
$$
\frac{|R \setminus S|}{16|R|} \leq \sum_{x \in R \setminus S} \prod_{i=1}^d p_i(x_i) \leq  \Hole_{(1,2,\ldots,d)}(S).
$$
Applying Claim \ref{clm:dHvsdD} and equation (\ref{telescope}), it follows that 
$$
\frac{|R \setminus S|}{16|R|} \leq D\left(p\mid \mid \prod_{i=1}^{d} p_{i}\right) = \sum_{i=2}^{d} D(p_{[i]}\mid \mid p_{[i-1]}p_i).
$$
Applying the bound (\ref{eq:SmallD}) gives
$$
\frac{|R \setminus S|}{16|R|} \leq 2 d \rho(\G) \eps,
$$
which implies that $|R \setminus S|  \leq 32 d \rho(\G) \eps |R|$. Hence, 
$|S| \geq (1-32 d \rho(\G) \eps)|R|$, and so 
\begin{equation}\label{eq:r-s} |R \setminus S| \leq \frac{32 d \rho(\G) \eps}{1- 32 d \rho(\G) \eps} |S| 
\leq 64 d \rho(\G) \eps |S|.\end{equation}
Combining (\ref{eq:s-r}) and (\ref{eq:r-s}) gives
$$
|S \triangle R| \leq (4d^2+64d)\rho(\G) \eps |S|,
$$
so we may take $B=R$, completing the proof.
\end{proof}

For completeness, we now present the deduction of Corollary \ref{corr:uniform-cover-stability} from Theorem \ref{thm:uniform-cover-stability}, using the approximation argument outlined in \cite{loomis1949}.

\begin{proof}[Proof of Corollary \ref{corr:uniform-cover-stability}]
Let $m \in \mathbb{N}$. Let $\G \subset \mathcal{P}([d])$ be a uniform $m$-cover. Let $S \subset \mathbb{R}^d$ be a body such that
$$ \mu(S) \geq (1-\eps) \left( \prod_{g \in \G} \mu(\pi_g (S)) \right)^{1/m}.$$

Fix $\eta \in (0,1)$. Since $\overline{S}$ is compact and $\mu(S) >0$, there exists a compact set $K \subset S$ such that $\mu(K) \geq (1-\eta)\mu(S)$. Choose an open cover $\mathcal{C}$ of $K$ by open cubes with corners at rational coordinates, such that all the cubes in $\mathcal{C}$ are contained within $S$. Since $K$ is compact, we may choose a finite subset $\mathcal{C}' \subset \mathcal{C}$ such that $\mathcal{C}'$ is a cover of $K$. Choose an axis-parallel grid of some side-length $\delta >0$, which is a common refinement of all the sets in $\mathcal{C}'$ (meaning that all the open cubes in $\mathcal{C}'$ are unions of open grid-cubes). Let $F$ be the union of all the open grid-cubes which are contained in $S$. Then $\mu(F) \geq (1-\eta)\mu(S)$. Let $N$ be the number of open grid-cubes in $F$, and for each $g \subset [d]$, let $N_{g}$ be the number of (lower-dimensional) open grid-cubes in the projection of $F$ onto the plane $\{x:\ x_i = 0\ \forall i \notin g\}$. Then we have
\begin{align*} N^m \delta^{md} & = \mu(F)^m\\
& \geq (1-\eta)^m \mu(S)^m\\
& \geq (1-\eta)^m (1-\epsilon)^m \prod_{g \in \G} \mu(\pi_g (S))\\
& \geq (1-\eta)^m (1-\epsilon)^m \prod_{g \in \G} N_{g} \delta^{|g|}\\
& = (1-\epsilon-\eta + \epsilon \eta)^m \delta^{m d} \prod_{g \in \G} N_{g}.
\end{align*}
Hence, cancelling the common factor of $\delta^{md}$ and rearranging, we obtain
$$N \geq (1-\epsilon - \eta + \epsilon \eta) \left(\prod_{g \in \G} N_{g}\right)^{1/m}.$$
Therefore, by Theorem \ref{thm:uniform-cover-stability}, there exists a box $B_{\eta} \subset \mathbb{R}^d$ (which is a union of open grid-cubes), such that 
$$\mu(F \Delta B_{\eta}) \leq b \rho(\mathcal{G})(\epsilon+\eta-\epsilon \eta)\mu(F).$$
Hence,
$$\mu(S \Delta B_{\eta}) \leq b \rho(\mathcal{G}) (\epsilon+\eta-\epsilon \eta)\mu(S) + \eta \mu(S).$$
Since $\eta \in (0,1)$ was arbitrary, it follows by a compactness argument that there exists a box $B \subset \mathbb{R}^d$ such that
$$\mu(S \Delta B) \leq b \rho(\mathcal{G}) \epsilon \mu(S).$$
\end{proof}

\subsection{Proofs of the main lemmas}\label{ProofsOfLemmas}
We now present the proofs of our main lemmas.
\begin{proof}[Proof of Lemma \ref{lem:TightToInfo}]
Our proof is information-theoretic, inspired by the technique of Radhakrishnan from \cite{radhakrishnan}.

Let $d,\G,\epsilon,S$ and $p$ be as in the statement of the lemma. Let $X$ be a random variable uniformly distributed on $S$. Assume without loss of generality that $i=d$. Recall that $H(X) = \log(|S|)$
and that by (\ref{eq:uniform-bound}), for all $g \subset [d]$, we have $H(X_g) \leq \log(|\pi_g (S)|)$. Set  
$m := m(\G)$, $\sigma := \sigma(\G)$ and $\rho:=\rho(\G)$.
Since 
$|S| \geq (1-\eps) \left( \prod_{g \in \G} |\pi_g (S)| \right)^{1/m}$, we have
\begin{align}\label{eq:4}
H(X) & = \log |S| \geq \log(1-\epsilon) + \frac{1}{m} \sum_{g \in \G} \log |\pi_g(S)| \geq \log(1-\eps) + \frac{1}{m} \sum_{g \in \G} H(X_g) \nonumber \\
& \geq - 2 \eps + \frac{1}{m} \sum_{g \in \G} H(X_g).
\end{align}
Hence,
\begin{comment}
\begin{align*}
\\
2 \eps &\geq \frac{1}{m} \sum_{g \in \G} H(X_g) - H(X) \tag{rearranging}\\
&=
 \frac{1}{m} \sum_{g \in \G} H(X_g) - H(X_d) -H(X_{[d-1]} \mid X_d) \tag{chain rule}\\
&= \frac{\sum_{g \in \G_{+d}} (H(X_d)+H(X_{g \setminus \{d\}} \mid X_d))+ \sum_{g \in \G_{-d}}H(X_g)}{m}
 - H(X_d) -H(X_{[d-1]} \mid X_d) \tag{chain rule}\\
 & = \frac{|\G_{+d}| H(X_d) + \sum_{g \in \G_{+d}} H(X_{g \setminus \{d\}} \mid X_d)+ \sum_{g \in \G_{-d}}H(X_g)}{m}
 - H(X_d) -H(X_{[d-1]} \mid X_d) \\
 & \geq \frac{\sum_{g \in \G_{+d}} H(X_{g \setminus \{d\}} \mid X_d)+ \sum_{g \in \G_{-d}}H(X_g)}{m}
 -H(X_{[d-1]} \mid X_d)
\tag{since $|\G_{+d}| \geq m$}\\
  &= \sum_{ j < d}\left(  \frac{\sum_{g \in \G_{+d,+j}} H(X_j \mid X_{g \cap [j-1]},X_d)+ \sum_{g \in \G_{-d,+j}}H(X_j \mid X_{g \cap [j-1]})}{m}
  -H(X_{j} \mid X_{[j-1]},X_d) \right) \tag{chain rule} \\
    &\geq  \sum_{ j < d}\left(  \frac{\sum_{g \in \G_{+d,+j}} H(X_j \mid X_{[j-1]},X_d)+ \sum_{g \in \G_{-d,+j}} H(X_j \mid X_{[j-1]})}{m}
  -H(X_{j} \mid X_{[j-1]},X_d) \right) \tag{since extra conditioning does not increase entropy}\\
  & = \frac{1}{m} \sum_{ j < d}\left((|\G_{+j}|-|\G_{-d,+j}|) H(X_j \mid X_{[j-1]},X_d)+|\G_{-d,+j}| H(X_j \mid X_{[j-1]})
  -mH(X_{j} \mid X_{[j-1]},X_d) \right)\\
  &\geq \frac{\sigma}{m}\sum_{ j < d} \left( H(X_j \mid X_{[j-1]}) - H(X_j \mid X_{[j-1]},X_d) \right) \tag{since $|\G_{+j}| \geq m$ and $|\G_{-d,+j}| \geq \sigma$}\\
 & = \rho^{-1} \left(H(X_{[d-1]}) - H(X_{[d-1]} \mid X_d)\right)\\
 & = \rho^{-1} I(X_{[d-1]};X_d),
\end{align*}
as required.
\end{comment}
\begin{align*}
2 \eps &\geq \frac{1}{m} \sum_{g \in \G} H(X_g) - H(X)\\
& = \frac{1}{m} \sum_{g \in \G} H(X_g) - H(X_d) -H(X_{[d-1]}\mid X_d)\tag{by the chain rule} \\
&= \frac{\sum_{g \in \G:\atop d \in g} (H(X_d)+H(X_{g \setminus \{d\}}\mid X_d))+ \sum_{g \in \G:\atop d \not \in g}H(X_g)}{m}
 - H(X_d) -H(X_{[d-1]}\mid X_d) \\
&\geq \frac{\sum_{g \in \G:\atop d \in g} H(X_{g \setminus \{d\}}\mid X_d)+ \sum_{g \in \G:\atop d \not \in g}H(X_g)}{m}
  -H(X_{[d-1]}\mid X_d) \tag{since, by definition of $m=m(\G)$, there are at least $m$ sets $g \in \G$ with $d \in g$} \\
 &= \sum_{ j < d}\left(  \frac{\sum_{g \in \G:\atop d,j \in g} H(X_j\mid X_{g \cap [j-1]},X_d)+ \sum_{g \in \G:\atop g \cap \{d,j\}=j}H(X_j\mid X_{g \cap [j-1]})}{m}
  -H(X_{j}\mid X_{[j-1]},X_d) \right) \tag{by the chain rule} \\
   &\geq  \sum_{ j < d}\left(  \frac{\sum_{g \in \G:\atop d,j \in g} H(X_j\mid X_{[j-1]},X_d)+ \sum_{g \in \G:\atop g \cap \{d,j\}=j}H(X_j\mid X_{[j-1]})}{m}
  -H(X_{j}\mid X_{[j-1]},X_d) \right) \tag{since extra conditioning does not increase entropy}\\
 &\geq \frac{\sigma}{m}\sum_{ j < d} \left( H(X_j\mid X_{[j-1]}) - H(X_j\mid X_{[j-1]},X_d) \right) \tag{using the definitions of $m$ and $\sigma$}\\
 &= \rho^{-1} \left(H(X_{[d-1]}) - H(X_{[d-1]} \mid X_d)\right) \tag{by the chain rule}\\
 &= \rho^{-1} I(X_{[d-1]};X_d),
\end{align*}
as required.
\end{proof}

\begin{remark}
As mentioned in the Introduction, Lemma \ref{lem:TightToInfo} can also be proved by appealing to an entropy inequality in the paper \cite{ballister-bollobas} of Balister and Bollob\'as. Indeed, without loss of generality taking $i=d$, it suffices to prove the inequality
$$\sigma H(X_d) + \sigma H(X_{[d-1]}) + (m-\sigma) H(X) \leq \sum_{g \in \G} H(X_g),$$
which follows from two applications of Theorem 6 in \cite{ballister-bollobas}, the first application being with $\mathcal{A} = \{g \in \G:\ d \notin \G\}$. We have opted to give the self-contained proof above, as it is not much longer. 
\end{remark}

\begin{proof}[Proof of Claim \ref{clm:dHvsdD}]
Recall that
$\log \frac{1}{1-\alpha} \geq \alpha$ for all $\alpha \in [0,1)$,
since
$1-\alpha \leq e^{-\alpha} \leq 2^{-\alpha}$. Using this, and the convexity of $t \mapsto \log (1/t)$, we have
\begin{align*}
D\left(p \mid \mid \prod_{j=1}^{r} p_{g_j}\right)& = - \sum_{x \in S} p(x)  \log \frac{\prod_{j=1}^{r} p_{g_j}(x_{g_j})}{p(x)} \\
&  \geq  - \log \sum_{x \in S} p(x) \frac{\prod_{j=1}^{r} p_{g_j}(x_{g_j})}{p(x)} \\
& =  \log \frac{1}{1- \sum_{x \not \in S} \prod_{j=1}^{r} p_{g_j}(x_{g_j})}\\
 & \geq \sum_{x \not \in S} \prod_{j=1}^{r} p_{g_j}(x_{g_j}) \\
 &= \Hole_{g_1,\ldots g_r}(S).
\end{align*}
\end{proof}

\subsubsection*{Trimming a 2-dimensional set to yield a rectangle}
In this subsection, we deal with the most  combinatorially-flavoured ingredient in our paper, the proof of Lemma \ref{lem:FindBox1}. Given a set $S \subset X_1 \times X_2$, we proceed to trim away some of its `vertical fibres' 
 (sets of the form $(\{x_1\} \times X_2) \cap S$, for some $x_1 \in X_1$), in such a way that
none of the remaining fibres is very short compared to the others. This leaves us with a subset $R_1 \subset X_1$, such that weight of each remaining vertical fibre of $R_1$ is at least a constant fraction of the average weight of the remaining fibres. The total amount of mass trimmed is bounded from above in terms of $\Hole(S)$. Applying this procedure with $\alpha=1/2$, and then repeating it to find a similar set $R_2 \subset X_2$, yields the following.
\begin{lemma}
\label{lem:RisZ2}
Let $X_1$ and $X_2$ be sets. If $S \subset X_1 \times X_2$ with $|S| < \infty$, then there exist $R_1 \subset X_1$ and $R_2 \subset X_2$ such that
$$|S \Delta (R_1 \times R_2)| \leq 20\, \Hole(S)\, |S|.$$
\end{lemma}

We will not use this lemma directly in our proof of our main theorem, but its proof is included for the reader's interest at the end of this subsection.
\begin{comment}
We omit the proof of this. In any case, in our applications, $X_2$ will be $(d-1)$-dimensional, and the set $R_1 \times R_2$ will usually not be a $d$-dimensional box.
\end{comment}

\begin{proof}[Proof of Lemma \ref{lem:FindBox1}]
 For $j \in \{1,2\}$, let $\pi_j(S)$ denote the natural projection of $S$ onto $X_j$. We find an appropriate set $R_1$ by iteratively removing from
  $\pi_1 (S)$ those points $x_1$ for which $p_1(x_1)$ is too `small',
  namely, those points whose fibre has size less than $(1-\alpha)$ times the average size of a fibre.

To this end, we define recursively:
$T^{(0)} := \emptyset$ and $R^{(1)}: = \pi_1 (S)$,
and for each $j \in \mathbb{N}$,
\begin{align*}
T^{(j)} & := \left\{ x_1 \in R^{(j)} \ : \
          \frac{p_1(x_1)}{p_1\left(R^{(j)}\right)}  \leq \frac
          {1-\alpha}{\left|  R^{(j)} \right|} \right\},\\
 R^{(j+1)} & := \pi_1(S) \setminus \bigcup_{r=1}^j T^{(r)}, \\
\eps^{(j)} & := p_1\left(T^{(j)} \right).
\end{align*}
Define the limit objects:
\begin{align*}
U^{(\infty)} & := \bigcup_{j= 1}^\infty T^{(j)}, \\
R_1 & := \pi_1(S) \setminus U^{(\infty)}, \\
\epsilon_1 & := p_1\left(  U^{(\infty)}\right) = \sum_{j=1}^\infty  \eps^{(j)} .
\end{align*}
Since $S$ is finite,
this process stabilizes after finitely many steps.
If we can show that 
\begin{equation}\label{eq:epsilon1bound} \eps_1 = p_1(\pi_1(S) \setminus R_1) \leq \frac{2\, \Hole(S)}{\alpha},\end{equation}
then by definition, for every $x_1 \in R_1$, we indeed have
$$p_1(x_1) \geq
p_1\left(R_1\right) \cdot \frac {1-\alpha}{\left|  R_1 \right|} 
\geq
\left(1-\frac{2 \Hole(S)}{\alpha}\right) \cdot \frac{(1-\alpha)}{|R_1|},$$  
proving Lemma \ref{lem:FindBox1}. 

\medskip To obtain
(\ref{eq:epsilon1bound}), it suffices to prove the following claim.
\begin{claim}\label{claim:SmallB1}
\ \ $\alpha \sum_{j=1}^\infty \eps^{(j)} \left(1 - \sum_{1 \leq r < j} \eps^{(r)}\right) \leq  \Hole(S).$
\end{claim}

Indeed, Claim \ref{claim:SmallB1} implies
\begin{align*}
\frac{1}{\alpha} \Hole(S) 
& \geq \sum_{j=1}^\infty \eps^{(j)} \left(1 - \sum_{1 \leq r < j} \eps^{(r)}\right) \\
& = \left( \sum_{j=1}^\infty \eps^{(j)} \right) - \left(\sum_{j=1}^\infty \sum_{1 \leq r < j} \eps^{(j)} \eps^{(r)}\right)
\geq \eps_1 - \frac{1}{2} \eps_1^2 \geq \frac{\eps_1}{2},
\end{align*}
implying (\ref{eq:epsilon1bound}).

It remains, therefore, to prove Claim~\ref{claim:SmallB1}.
\newline

\noindent {\it Proof of Claim~\ref{claim:SmallB1}.}
 Fix $j \in \mathbb{N}$ and let $x_1 \in T^{(j)}$. Let 
$$
Y = Y(x_1) := \{y \in \pi_2 (S) : (x_1,y) \in S\}
$$
be the fibre of $x_1$,  and let $\overline{Y} = \overline{Y}(x_1) =
\pi_2(S) \setminus Y$. We first bound the mass of the set of elements of $S$ whose projection
lies in $Y$. 
  
By definition, we have
$$p_1(R^{(j)}) = 1 - \sum_{r<j}\eps^{(r)}$$
and
$$p_1(x_1) \leq \frac{(1-\alpha) p_1(R^{(j)})}{|R^{(j)}|}.$$
So
$$
\left|R^{(j)} \times Y  \right| = \left|R^{(j)}\right|\left|Y\right| = \left| R^{(j)}\right| \cdot p_1(x_1) |S| 
\leq (1-\alpha) |S|  \left( 1- \sum_{r<j}\eps^{(r)} \right) .
$$
We therefore have
\begin{align*}
\sum_{y \in Y} p_2(y) 
& = \sum_{y \in Y} \sum_{x'_1 \in \pi_1 (S)} p(x'_1,y) \\
& \leq \frac{\left| R^{(j)} \times Y \right|}{|S|} +\sum_{y \in Y} 
 \sum_{x'_1 \not \in R^{(j)}}  p(x'_1,y)  
\tag{since $p$ is uniform on $S$} \\
& \leq (1-\alpha) \left( 1- \sum_{r<j}\eps^{(r)} \right) + \sum_{x'_1 \notin R^{(j)}} \sum_{y \in \pi_2(S)} p(x'_1,y)\\
& = (1-\alpha) \left( 1- \sum_{r<j}\eps^{(r)} \right)  + \sum_{r<j}\eps^{(r)}\\
& = 1-\alpha + \alpha \sum_{r < j} \eps^{(r)}.
\end{align*}
It follows that
\begin{align}\label{eq:1}
\sum_{y \in \overline{Y}} p_2(y) 
& \geq 1 - \left(1-\alpha + \alpha \sum_{r < j} \eps^{(r)} \right)  
=  \alpha \left( 1 - \sum_{r<j}\eps^{(r)} \right) .
\end{align}
Now summing over $x_1$ and using \eqref{eq:1}, we have 
$$
\sum_{x_1 \in T^{(j)}}p_1(x_1) \sum_{ y \in \overline{Y}(x_1)} p_2(y) \geq \eps^{(j)} \alpha \left( 1 - \sum_{r<j}\eps^{(r)} \right) .
$$
Finally, summing over all $j$, and recalling that if $y \in \overline{Y}(x_1)$ then $(x_1,y) \not \in S$, we have
$$
\Hole(S) = \sum_{(x_1,y) \not \in S} p_1(x_1) p_2(y) \geq \alpha
\sum_{j=1}^\infty \eps^{(j)}  \left( 1 - \sum_{r<j}\eps^{(r)} \right) .
$$
This completes the proof of Claim \ref{claim:SmallB1}, and thus of Lemma \ref{lem:FindBox1}.
\end{proof}

\begin{proof}[Proof of Lemma \ref{lem:RisZ2}.]
  We may assume that $\Hole(S) < 1/20$, otherwise the conclusion of the lemma holds with $R_1=R_2=\emptyset$. Let $R_1 \subset X_1$ be the set given by
  Lemma~\ref{lem:FindBox1} with $\alpha = 1/2$. Let
  $R_2 \subset X_2$ be the set given by Lemma~\ref{lem:FindBox1}
  with $\alpha=1/2$, when coordinates 1 and 2 are interchanged. Let $R = R_1 \times R_2$.  By the union bound, we have
  \begin{equation}
    \label{eq:2}
\frac{|S \setminus R|}{|S|} \leq 2\cdot \frac{2\, \Hole(S)}{\alpha} = 8\, \Hole(S).    
  \end{equation}
On the other hand, for every $x \in R$, %\gnote{ was $R\setminus S$} 
we have
$$p_1(x_1)p_2(x_2) \geq 
\left(1-\frac{2\, \Hole(S)}{\alpha}\right)^2 (1-\alpha)^2 \frac{1}{|R|}
> \frac{1}{7 |R|}.$$
Therefore,
\begin{align*}
\Hole(S)
\geq \sum_{x \in R \setminus S} 
p_1(x_1)p_2(x_2) \geq |R \setminus S| \cdot 
\frac{1}{7 |R|} ,
\end{align*}
which implies that $|R \setminus S|  \leq 7\, \Hole(S)\, |R|$. It follows that
$$|S| \geq |S\cap R| \geq (1-7\, \Hole(S))|R| \geq 13|R|/20,$$
and therefore
\begin{equation}
  \label{eq:3}
| R \setminus S| \leq 12\, \Hole(S)\, |S|.  
\end{equation}
Combining \eqref{eq:2} and \eqref{eq:3}, we have $|S \triangle R| \leq 20\, \Hole(S)\, |S|$, completing the proof.
\end{proof}

\section{A stability result for the edge-isoperimetric inequality in $\mathbb{L}^d$}
\label{sec:edgeIso}
In this section, we prove Theorem \ref{thm:edge-iso-stability},
our stability result for the edge-isoperimetric inequality in 
the lattice $\mathbb{L}^d$, using Theorem \ref{thm:uniform-cover-stability} and some additional combinatorial arguments.

We start with a short proof of the isoperimetric inequality (Theorem \ref{thm:edge-iso}), as it will be useful to refer to it later.

\begin{proof}[Proof of Theorem \ref{thm:edge-iso}.]
We write $\partial_i(S)$ for the set of edges in $\partial S$ of the form $\{x,x+e_i\}$, i.e. the set of all direction-$i$ edges of $\partial S$. Since $|S| < \infty$, for each $x \in S$, there are at least two direction-$i$ edges of $\partial S$ which project to $\pi_{[d] \setminus \{i\}}(x)$, and therefore
$$|\partial_i(S)| \geq 2|\pi_{[d] \setminus \{i\}} (S)|.$$
Summing over all $i$, we obtain
$$|\partial S| = \sum_{i=1}^{d}|\partial_i(S)| \geq 2 \sum_{i=1}^{d} |\pi_{[d] \setminus \{i\}} (S)|.$$
The AM-GM inequality and the Loomis-Whitney inequality yield
\begin{equation}\label{eq:am-gm}
|\partial S| \geq 2 \sum_{i=1}^{d} |\pi_{[d] \setminus \{i\}} (S)| \geq 2d \left(\prod_{i=1}^{d} |\pi_{[d] \setminus \{i\}} S|\right)^{1/d} \geq 2d|S|^{(d-1)/d}.
\end{equation}
\end{proof}

The following easy stability result for the AM-GM inequality will be useful in our proof of Theorem \ref{thm:edge-iso-stability}.

\begin{prop}
\label{prop:AM-GM-stability}
Let $0 \leq \eps \leq 1/(16 d)$. Let $z_1 \geq z_2 \geq \ldots \geq z_d > 0$ be such that
$$\frac{1}{d} \sum_{i=1}^{d} z_i \leq (1+\eps) \left( \prod_{i=1}^{d} z_i \right)^{1/d}.$$
Let $G := \left( \prod_{i=1}^{d} z_i\right)^{1/d}$. Then
$$\forall i \in [d],\qquad (1-2d\sqrt{d\epsilon}) G \leq z_i \leq
(1+2\sqrt{d\epsilon}) G.$$
\end{prop}
\begin{proof}
First, we assert that
\begin{equation} \label{eq:assertion} z_1 \leq (1+4\sqrt{d\epsilon})G.\end{equation}
Let $\eta \geq 0$ be such that $z_1 = (1+\eta) G$,
and let $\eta_0 := 4\sqrt{d\epsilon}$; then $\eta_0 \leq 1/2$.
Assume for a contradiction that $\eta > \eta_0$. Then
\begin{align*} 
(1+\epsilon) Gd 
& \geq \sum_{i=1}^{d} z_i \\
& \geq z_1 + (d-1) \left( \prod_{i=2}^d z_i \right)^{1/(d-1)} \tag{AM-GM} \\
& = z_1 + (d-1) G^{d/(d-1)} z_1^{-1/(d-1)}  \\
& = G \left( 1+\eta + (d-1)  (1+\eta)^{-1/(d-1)} \right).
\end{align*}
The function $f:\eta \mapsto \eta + (d-1)(1+\eta)^{-1/(d-1)}$ has
$f'(\eta) >0 $ for all positive $\eta$, and is therefore strictly increasing on $[0, \infty)$. Hence,
\begin{align*} 
(1+\epsilon) d & \geq 1+\eta + (d-1)  (1+\eta)^{-1/(d-1)}\\
& > 1+\eta_0 + (d-1)  (1+\eta_0)^{-1/(d-1)} \\
& \geq d + \frac{d}{2(d-1)}\eta_0^2 - \frac{d(2d-1)}{6(d-1)^2} \eta_0^3 \\
& \geq d + \left(\frac{d}{2(d-1)}\eta_0^2 - \frac{d(2d-1)}{12(d-1)^2}\right)\eta_0^2\\ 
& \geq d+\tfrac{1}{4}\eta_0^2 ,
\end{align*}
contradicting the definition of $\eta_0$.

It follows from (\ref{eq:assertion}) that 
$$Gd \leq
\sum_{i=1}^{d} z_i
\leq (d-1)(1+\eta_0) G + z_d,$$
so
\begin{align*} 
z_d \geq G(1-(d-1)\eta_0) \geq G(1-d\eta_0),
\end{align*}
completing the proof.
\end{proof}

\begin{proof}[Proof of Theorem~\ref{thm:edge-iso-stability}]
Let $S \subset \Z^d$ with $|S| < \infty$ and
\begin{equation}\label{eq:small-boundary} |\partial S| \leq 2d |S|^{(d-1)/d}(1+\epsilon).\end{equation}
%Write $|S| = a^{d}$, where $a >0$. 
%If $\epsilon > 1/(3492^2 d^{11})$, then we can take $C = \emptyset$; so 
We may assume that
\begin{equation}\label{eq:epsilon-assumption} 
\epsilon \leq \frac{1}{72^2 d^{5}},
\end{equation}
otherwise the conclusion of the theorem holds trivially with $C = \emptyset$.

By \eqref{eq:am-gm} and (\ref{eq:small-boundary}), we have
\begin{align}
\label{eq:SvsG}
\left(\prod_{i=1}^{d} |\pi_{[d] \setminus \{i\}} S|\right)^{1/(d-1)} \leq \left(\frac{|\partial S|}{2d}\right)^{d/(d-1)} \leq (1+\epsilon)^{d/(d-1)} |S| \leq \frac{1}{1-2\epsilon}|S|.
\end{align}
Corollary \ref{cor:lw-stability} now implies that there exists a box 
$$R = R_1 \times R_2 \times \ldots \times R_d \subset \mathbb{Z}^d$$ such that
\begin{equation}
\label{eq:close}
|S \triangle R| \leq 2(64d+4d^2)(d-1) \epsilon |S| \leq 72d^3\epsilon |S|
\leq \delta |S|,\end{equation}
where
$$\delta:= \sqrt {d \epsilon}.$$
%210^2 d^3 \eps < 1
%22^2 d^3 \eps < 1/4
Our aim is to show that $R$ is close in symmetric difference to a cube.

Clearly, by (\ref{eq:close}), we have
\begin{equation}\label{eq:sizes-close} (1-\delta)|S| \leq |R| \leq (1+\delta)|S|.\end{equation}
Let
$$G: = \left( \prod_{i \in [d]} |\pi_{[d] \setminus \{i\}} S| \right)^{1/d}
\quad \text{and} \quad a: = |S|^{1/d}.$$
Note that, by (\ref{eq:SvsG}) and the Loomis-Whitney inequality, we have
\begin{equation} \label{eq:avsG} a^{d-1} = |S|^{(d-1)/d} \leq G \leq (1+\epsilon) |S|^{(d-1)/d} = (1+\epsilon)a^{d-1}.\end{equation}

By \eqref{eq:am-gm} and (\ref{eq:small-boundary}), we have
$$ \frac{1}{d} \sum_{i=1}^{d} |\pi_{[d] \setminus \{i\}} (S)| \leq \frac{|\partial S|}{2d} \leq (1+\epsilon) |S|^{(d-1)/d} \leq (1+\epsilon)\left( \prod_{i \in [d]} |\pi_{[d] \setminus \{i\}} (S)| \right)^{1/d}.$$

Hence, by Proposition \ref{prop:AM-GM-stability}, all the $(d-1)$-dimensional projections of $S$ are of roughly equal size:
\begin{align}
\label{eqn:ConcStabAMGM}
\forall i \in [d] ,\qquad (1-2 d \delta ) G \leq |\pi_{[d] \setminus \{i\}} (S)| \leq
(1+2\delta) G.\end{align}
We now show that all the 1-dimensional projections of $R$ are of roughly equal size.

\begin{claim}
\label{claim:approx-size1}
For every $i \in [d]$,
we have $(1-7 \delta) a \leq |R_i| \leq (1+ 14 d \delta)  a$.
\end{claim}

\begin{proof}
Without loss of generality, we prove the claim for $i=1$.
Observe that
\begin{equation}\label{eq:projection-sizes} (1- 2 \delta) |\pi_{\{2,\ldots,d\}} (R)|
\leq |\pi_{\{2,\ldots,d\}} (S)|.\end{equation}
Indeed, if this does not hold, then, using (\ref{eq:sizes-close}), we have
$$|S \triangle R| \geq |R \setminus S| >  2 \delta |R| \geq 2(1-\delta)|S| \geq \delta |S|,$$
contradicting (\ref{eq:close}). Therefore, using (\ref{eq:sizes-close}), (\ref{eq:projection-sizes}), \eqref{eqn:ConcStabAMGM} and (\ref{eq:avsG}) successively, we have
\begin{align*} |R_1| & = \frac{|R|}{|\pi_{\{2,3,\ldots,d\}} (R)|} \\
& \geq \frac{(1- 2 \delta)(1-\delta)|S|}{|\pi_{\{2,\ldots,d\}}(S)|} \\
& \geq \frac{(1- 2 \delta)(1-\delta)|S|}{(1+ 2 \delta)G} \\
& \geq \frac{(1- 2 \delta)(1-\delta)|S|}{(1+ 2\delta)(1+\epsilon)|S|^{(d-1)/d}} \\
& \geq (1-7 \delta)|S|^{1/d}\\
& = (1-7 \delta)a.
\end{align*}
Similarly, we have
\begin{equation} \label{eq:Rlarge} \forall i \in [d],\qquad |R_i| \geq  (1-7 \delta)a.\end{equation}
Hence, using (\ref{eq:sizes-close}) and (\ref{eq:Rlarge}) we obtain
\begin{align*} |R_1| &= \frac{|R|}{\prod_{i=2}^{d} |R_j|}\\
& \leq \frac{(1+\delta)|S|}{(1-7\delta)^{d-1} |S|^{(d-1)/d}}\\
& \leq (1+14 d\delta)|S|^{1/d}\\
& = (1+14 d\delta)a.
%&\\
%& = \frac{(1+\delta_0)a^{d}}{((1-7\delta_0)a)^{d-1}} \\
%& \leq \frac{1+\delta_0}{1-7(d-1)\delta_0}a \\
%& \leq (1+\delta_0) (1+14(d-1)\delta_0)a \leq (1+ 16d\delta_0)a,
\end{align*}
%using the fact that $7(d-1)\delta_0 \leq 1/2$. This completes the proof of Claim \ref{claim:approx-size}.
\end{proof}

Next, for each $i \in [d]$ we throw away the elements $c \in R_i$ such that $\{x \in R:\ x_i = c\}$ contains few elements of $S$, producing a slightly smaller box $R'$.

Fix $i \in [d]$; without loss of generality, $i=1$. Define 
$$Q: = R_2 \times R_3 \times \ldots \times R_d.$$ 

For each $c \in R_1$, call $c$ {\em heavy} if $|S \cap \{x \in R:\ x_1=c\}| \geq 7 |Q| /8$, and call $c$ {\em light} otherwise. We assert that there are at most $16 \delta |R_1|$ light elements in $R_1$. Indeed, if the number of light elements is larger than this, then
 $$|S \triangle R| \geq |R \setminus S| > (|Q|/8) \cdot 16\delta  |R_1| = 2 \delta |R| \geq 2\delta(1-\delta) |S| \geq \delta |S|,$$
contradicting (\ref{eq:close}).

Let $R_1' : = \{c \in R_1:\ c \textrm{ is heavy}\}$. Then $|R_1 \setminus R_1'| \leq 16\delta|R_1|$. Define $R_i'$ similarly for each $i \in \{2,3,\ldots,d\}$, and let
$$R' = R_1' \times R_2' \times \ldots \times R_d'.$$

By the union bound, we have
\begin{equation} \label{eq:R'R} |R \setminus R'| \leq 16d\delta |R|,\end{equation}
and therefore
\begin{equation}\label{eq:R'large} |R'| \geq (1-16d\delta)|R|.\end{equation}

Our next step is to show that, for each $i \in [d]$, $R_i'$ occupies most of the interval in which it is contained, as if not, $|\partial S|$ would be too large.

\begin{claim}
\label{claim:close-to-interval1}
For each $i \in [d]$, 
we have $\max(R_i') - \min(R_i')+1 - |R_i| \leq 8\delta |R_i|$.
\end{claim}

\begin{proof}
Without loss of generality, we prove the claim for $i=1$. 
Let 
$$u := \min(R_1'), \quad v := \max(R_1').$$ 
Suppose for a contradiction that
\begin{equation} \label{eq:too-wide} v - u+1 - |R_1| > 8 \delta |R_1|.\end{equation}
 
Define
\begin{align*} 
U & := \pi_{\{2,3,\ldots,d\}}(S \cap \{x \in R:\ x_1=u\}),\\
V & := \pi_{\{2,3,\ldots,d\}}(S \cap \{x \in R:\ x_1=v \}),\\
F & := \{y \in Q:\ (x_1,y) \in S\ \forall x_1 \in \{u,u+1,\ldots,v-1,v\}\};
\end{align*}
note that $U,V,F \subset Q$. Since $|S| < \infty$, for each $z \in U$, there is at least one edge in $\partial_1(S)$ of the form $\{(x_1,z),(x_1+1,z)\}$, where $x_1 < u$. Similarly, for each $z \in V$, there is at least one edge in $\partial_1(S)$ of the form $\{(x_1,z),(x_1+1,z)\}$, where $x_1 \geq v$. Moreover, for each $z \in (U \cap V) \setminus F$, there is at least one edge in $\partial_1(S)$ of the form $\{(x_1,z),(x_1+1,z)\}$ where $u \leq x_1 < v$. It follows that
\begin{align}
\label{eq:extra-boundary} |\partial_1(S)| & \geq |U| + |V| + |(U \cap V)\setminus F| \nonumber \\
& = |U| + |V| + |U \cap V| - |F| \nonumber \\
 &\geq \tfrac{7}{8}|Q| + \tfrac{7}{8}|Q| + \tfrac{3}{4}|Q| - |F| \nonumber \\
 &= \tfrac{5}{2}|Q|-|F|.
\end{align}

We now assert that
\begin{equation}\label{eq:F-bound} |F| \leq \tfrac{1}{4}|Q|.\end{equation}
Indeed, suppose on the contrary that \eqref{eq:F-bound} does not hold.
By (\ref{eq:too-wide}), the interval $\{u,u+1,\ldots,v\}$ contains more than $8 \delta |R_1|$ elements not in $R_1$.
Therefore,
$$|S \setminus R| > 8 \delta |R_1| \cdot |F| > 8 \delta |R_1| \cdot \tfrac{1}{4}|Q| = 2\delta |R| \geq \delta|S|,$$
contradicting (\ref{eq:close}).

Using Claim \ref{claim:approx-size1} and (\ref{eq:close}), we have
\begin{equation} \label{eq:Q-lower-bound} 
|Q| = \frac{|R|}{|R_1|} \geq \frac{(1-\delta)a^d}{(1+ 14 d\delta)a} \geq (1-15 d\delta) a^{d-1} .
\end{equation}

Hence, by~\eqref{eq:F-bound} and~\eqref{eq:extra-boundary}, we have
$$|\partial_1(S)| \geq 2|Q| + \tfrac{1}{4}|Q| \geq (2+\tfrac{1}{4})(1-15 d\delta) a^{d-1}.$$
But we also have
$$\forall i \in \{2,3,\ldots,d\}, \qquad |\partial_i(S)| \geq 2|\pi_{[d]\setminus \{i\}}(S)| \geq 2(1-2d\delta)G \geq 2(1-2d\delta)a^{d-1},$$
using (\ref{eq:avsG}) and (\ref{eqn:ConcStabAMGM}). Hence, using (\ref{eq:epsilon-assumption}), we have
\begin{align*} |\partial S| & = \sum_{i=1}^{d}|\partial_i(S)|\\
& \geq (2+\tfrac{1}{4})(1-15 d\delta)a^{d-1} + (d-1)\cdot 2(1-2d\delta)a^{d-1}\\
& \geq 2da^{d-1}(1+\tfrac{1}{16d}-2d\delta) \\
& > 2da^{d-1}(1+\epsilon) ,\end{align*}
contradicting (\ref{eq:small-boundary}).
\end{proof}

We can now complete the proof of the theorem.
Let
$$L = \max\{\max(R_i') - \min(R_i') + 1:\ i \in [d]\}.$$
Assume without loss of generality that
$$L = \max(R_1') - \min(R_1') + 1.$$
Define
$$C =  \prod_{i=1}^{d} 
\{\min(R_i'),\min(R_i')+1,\ldots, \min(R_i') + L-1\}.$$
The set $C$ is a cube of side-length $L$ containing $R'$.
Using Claims~\ref{claim:approx-size1}
and~\ref{claim:close-to-interval1}, we have
\begin{align*} L-a 
& = \max(R_1') - \min(R_1') + 1 - |R_1| + (|R_1| - a)\\
& \leq 8\delta (1+14d\delta) a + 14d\delta a\\
& \leq 20 d \delta a.
\end{align*}
Therefore, using (\ref{eq:R'large}) and (\ref{eq:sizes-close}), we have
\begin{align*} 
|R' \triangle C| 
& = L^d - |R'|\\
& \leq (1+20d\delta)^d a^d - (1-16d\delta)(1-\delta)a^d\\
& \leq (1+40d^2\delta)a^d - (1-17d\delta)a^d\\
& \leq 50d^2 \delta a^d.
\end{align*}
Finally, using (\ref{eq:R'R}), (\ref{eq:close}) and (\ref{eq:sizes-close}), we obtain
\begin{align*} |S \triangle C| & \leq |S \triangle R| + |R \triangle R'| + |R' \triangle C|\\
& \leq \delta a^d + 16d\delta (1+\delta)a^d + 50d^2 \delta a^d\\
& \leq 60d^2 \delta a^d\\
& = 60 d^{5/2} \sqrt{\epsilon} |S|,
\end{align*}
proving the theorem.
\end{proof}

\begin{remark}
\label{remark:iso-sharpness}
As observed in the Section~\ref{sec:intro}, Theorem \ref{thm:edge-iso-stability} is sharp up to a constant factor depending upon $d$ alone. To see this, take $S = [a]^{d-1} \times [b]$, where $b = (1+\phi) a$, $\phi >0$ and $a,b \in \mathbb{N}$, i.e. $S$ is a `cuboid'. Then
$$|S| = a^{d-1}b = (1+\phi)a^d,$$
and 
$$|S \triangle C| \geq (b-a) a^{d-1} = \phi a^{d} = \frac{\phi}{1+\phi}|S|$$
for all cubes $C \subset \mathbb{Z}^d$. On the other hand, we have
\begin{align*}
|\partial S| - 2d|S|^{(d-1)/d} & = 2a^{d-1} + 2(d-1) a^{d-2}b - 2d (a^{d-1}b)^{(d-1)/d}\\
& = 2a^{d-1}(1+(d-1)(1+\phi) - d(1+\phi)^{(d-1)/d})\\
& \leq 2a^{d-1}(d+\phi(d-1) - d[1+\tfrac{d-1}{d} \phi - \tfrac{d-1}{2d^2} \phi^2]) = \tfrac{d-1}{d} \phi^2 a^{d-1}.
\end{align*}
Hence, we have $|\partial S| = (1+\epsilon)2d|S|^{(d-1)/d}$, where 
$$\epsilon = \frac{|\partial S| - 2d|S|^{(d-1)/d}}{2d|S|^{(d-1)/d}} \leq \frac{\tfrac{d-1}{d} \phi^2 a^{d-1}}{2da^{d-1} (1+\phi)^{(d-1)/d}} \leq \frac{\phi^2}{2d},$$
but $S$ is $\delta$-far from any cube, where 
$$\delta = \frac{\phi}{1+\phi} \geq \tfrac{1}{2}\phi \geq \tfrac{1}{2} \sqrt{2d\epsilon}.$$
\end{remark}

\section{Conclusion and open problems}\label{sec:conclusion}

We have proved stability results for the Uniform Cover inequality (Theorem~\ref{thm:uniform-cover-stability}), and
the edge-isoperimetric inequality in the $d$-dimensional
integer lattice (Theorem~\ref{thm:edge-iso-stability}).

We conjecture that the dependence on $d$ in Theorem \ref{thm:uniform-cover-stability} can be removed:

\begin{conjecture}
Let $d \in \mathbb{N}$ with $d \geq 2$. Let $S \subset \mathbb{Z}^d$ with $|S| < \infty$. If
$$|S| \geq (1-\eps) \left( \prod_{g \in \G} |\pi_g S| \right)^{1/m(\G)},$$
then there exists a box $B \subset \mathbb{Z}^d$ such that
$$|S \Delta B| \leq c \rho(\G) \eps\,|S|,$$
where $c$ is an absolute constant.
\end{conjecture}

We also conjecture that the dependence upon $d$ in Theorem \ref{thm:edge-iso-stability} can be improved to $\Theta(\sqrt{d})$:
\begin{conjecture}
Let $S \subset \mathbb{Z}^d$ with $|S| < \infty$ and with
$$|\partial S| \leq 2d |S|^{(d-1)/d}(1+\epsilon).$$
Then there exists a cube $C \subset \mathbb{Z}^d$ such that $|S \Delta C| \leq c' \sqrt{d\epsilon}\, |S|$, where $c'$ is an absolute constant.
\end{conjecture}
This would be sharp up to the value of $c'$, by the cuboid example in Remark~\ref{remark:iso-sharpness}.

%%% AUTHOR: optional appendix here
%\appendix %% you may comment this out if no Appendix
%\section*{Appendix}
%\section{Improving the constants}
%Material is placed here as needed.

%%% AUTHOR: optional acknowledgments here
\section*{Acknowledgments} %%  you may comment this out if no Ackno
The authors would like to thank Itai Benjamini and Emanuel Milman for helpful conversations and comments, and an anonymous referee for several helpful suggestions and remarks.

%%% AUTHOR:
%%% Bibliography goes here. Note that the arXiv cannot process bibtex
%%% or biber bibliographies.  Example of acceptable bibliograpy format:
%\bibliographystyle{abbrv}
%\bibliography{daj-geom-stability}
%\input{daj-geom-stability.bbl}
%% AUTHOR: You can generate such a bibliography from a .bib file by 
%% running pdflatex/bibtex/pdflatex/pdflatex and then pasting the .bbl file
%% between 

%%% AUTHOR: Include a short description of each author following the
%%% structure below. Use the same short tags used previously.  
%%% Use \imageat{} and \imagedot{} instead of "@" and "." in
%%% email addresses-this replaces the symbols with graphics to avoid 
%%% e-mail address harvesting from the .pdf file
\begin{dajauthors}
\begin{authorinfo}[david]
  David Ellis\\
  School of Mathematical Sciences,\\
  Queen Mary, University of London,\\
  Mile End Road,\\
  London E1 4NS,\\
  United Kingdom.\\
  d\imagedot{}ellis\imageat{}qmul\imagedot{}ac\imagedot{}uk
\end{authorinfo}
\begin{authorinfo}[ehud]
  Ehud Friedgut\\
  Department of Mathematics,\\
  Weizmann Institute of Science,\\
  Rehovot 76100,\\
  Israel.\\
  ehud.friedgut\imageat{}gmail\imagedot{}com
\end{authorinfo}
\begin{authorinfo}[guy]
  Guy Kindler\\
  School of Computer Science and Engineering,\\
  The Hebrew University of Jerusalem,\\
  Edmond J. Safra Campus,\\
  Jerusalem 91904,\\
  Israel.\\
  gkindler\imageat{}cs\imagedot{}huji\imagedot{}ac\imagedot{}il
\end{authorinfo}
\begin{authorinfo}[amir]
  Amir Yehudayoff\\
  Department of Mathematics,\\
  Technion-IIT,\\
  Haifa 32000,\\
  Israel.\\
  amir.yehudayoff\imageat{}gmail\imagedot{}com
\end{authorinfo}
\end{dajauthors}

\end{document}